\documentclass[12pt]{article}
\usepackage{amscd}
\usepackage{amsmath,amssymb,mathrsfs}
\usepackage{hyperref}
\usepackage{xcolor}
\usepackage[all]{xypic}
\usepackage{times}
\usepackage{inputenc}
\usepackage{titlesec}
\usepackage{lipsum}
\newcommand\blfootnote[1]{%
  \begingroup
  \renewcommand\thefootnote{}\footnote{#1}%
  \addtocounter{footnote}{-1}%
  \endgroup
}

\def\inf{\operatorname{inf}}

 \usepackage{amsthm}
 \newtheorem{lemma}{Lemma}[section]
 \newtheorem{corollary}[lemma]{Corollary}
 
 \newtheorem{theorem}[lemma]{Theorem}
 \newtheorem{proposition}[lemma]{Proposition}
 \newtheorem{definition}[lemma]{Definition}
 \newtheorem{remark}[lemma]{Remark}
 \newtheorem{example}[lemma]{Example}

\advance\headheight1.15pt

\begin{document}

\title{\fontsize{15}{0}\selectfont 
Trace Principle for Riesz Potentials on Herz-Type Spaces and Applications
}
\author{\fontsize{11}{0}\selectfont 
M. Ashraf Bhat and G. Sankara Raju Kosuru$^*$ \\
\fontsize{10}{0}\textit{{Department of Mathematics, Indian
Institute of Technology Ropar, Rupnagar-140001, Punjab, India.}}
}
\date{}
\maketitle

\thispagestyle{empty}

\begin{center}
\textbf{\underline{ABSTRACT}}
\end{center}
We establish trace inequalities for Riesz potentials on Herz-type spaces and discuss the optimality of conditions imposed on specific parameters. We also present some applications in the form of Sobolev-type inequalities, including the Gagliardo-Nirenberg-Sobolev inequality and the fractional integration theorem in the Herz space setting. In addition, we obtain a Sobolev embedding theorem for Herz-type Sobolev spaces.

\textit{Keywords}\textbf{:} Riesz potential; Adams' trace inequality; Herz spaces; Sobolev inequalities.
\blfootnote{$^{*}$Corresponding Author} 
\blfootnote{G. Sankara Raju Kosuru: raju@iitrpr.ac.in} \blfootnote{%
Mohd Ashraf Bhat: ashraf74267@gmail.com} 

2020 \textit{Mathematics Subject Classification:} 31C15, 26A33, 42B35, 46E35.\vspace{1cm}


\section{Introduction and Preliminaries}

The Riesz potential operator $I_{\gamma}$ is an
integral operator defined by the convolution of a function $f$ with the Riesz kernel $K_{\gamma}(x):=|x|^{\gamma-n}$. More precisely, for $n \in \mathbb{N}$ and $0 < \gamma < n$,
$$I_{\gamma}f(x):=\int_{\mathbb{R}^n}\frac{f(y)}{|x-y|^{n-\gamma}}dm(y)~;~x \in \mathbb{R}^n.$$
Here, $f$ is a suitable function, such as a locally integrable function on $\mathbb{R}^n$ $( L^1_{loc}(\mathbb{R}^n) )$, or a function with sufficiently rapid decay at infinity, particularly if $f \in L^p(\mathbb{R}^n)$ with $1 \leq p < \frac{n}{\gamma}$, and $m$ is the Lebesgue measure on $\mathbb{R}^n$. If $\gamma = 2 \neq n$, this integral operator is called the Newtonian potential and is used to describe the potential energy distribution of a system of point masses in classical mechanics or the electrostatic potential created by a charge distribution in physics.

The trace problem for Riesz Potentials deals with finding non-negative (positive) Borel measures $\mu$ on $\mathbb{R}^n$ such that $I_{\gamma}$ maps $\mathcal{F}(\mathbb{R}^n, m)$ boundedly into $\mathcal{F}^{\prime}(\mathbb{R}^n, \mu)$, where $\mathcal{F}(\mathbb{R}^n,m)$ and $\mathcal{F}^{\prime}(\mathbb{R}^n,\mu)$ are function spaces defined over $\mathbb{R}^n$ with respect to measures $m$ and $\mu$ respectively. 
Adams \cite{MR336316, MR287301}  proved that for $1<p_1<p_2<\infty$ and $0<\gamma<\frac{n}{p_1}$, 

\begin{equation}\label{adams-Inequality}
\left \Vert I_{\gamma}f \right \Vert_{L^{p_2}(\mathbb{R}^n, \mu)} \lesssim \left \Vert f \right \Vert_{L^{p_1}(\mathbb{R}^n, m)}
\end{equation}
if and only if $\mu(B) \lesssim [m(B)]^{p_2\left(\frac{1}{p_{1}}-\frac{\gamma}{n}\right)}$ for every ball $B \subset \mathbb{R}^n$.
Here we have used the standard notation $\zeta \lesssim \rho$ (or equivalently $\rho\gtrsim \zeta$) to express that there exists a positive constant $c$, independent of relevant variables, such that $\zeta \leq c \rho$. The inequality (\ref{adams-Inequality}) is not true when $p_1 = p_2$ (see, for example, \cite{MR417774}). Inequalities involving Riesz potentials often provide an important tool for estimating functions in terms of the norms of their derivatives. The wide-ranging applicability of trace inequalities for Riesz potentials has sparked significant interest in recent studies, for instance, see  \cite{MR4327616, MR2028180, MR3718049, MR3836125} and references therein. In the realm of Morrey-Lorentz spaces, the following theorem has been established \cite{MR4327616}.
  \begin{theorem} \label{AI-Morrey}
  Let $1<p_1 \leq q_1 < \infty$ and $1<p_2 \leq q_2< \infty$ satisfy $\frac{p_2}{q_2} \leq \frac{p_1}{q_1}$ for all $1<p_1<p_2<\infty$. Then, the inequality
  $$ \left \Vert I_{\gamma}f \right \Vert_{\mathcal{M}_{p_2,r_2}^{q_2}(\mathbb{R}^n, \mu)} \lesssim \left \Vert f \right \Vert_{\mathcal{M}_{p_1,r_1}^{q_1}(\mathbb{R}^n, m)}$$
  holds true if and only if the measure $\mu$ satisfies $ \mu(B) \lesssim [m(B)]^{q_2\left(\frac{1}{q_1}-\frac{\gamma}{n}\right)}$ for every ball $B \subset \mathbb{R}^n$, given that $~n\left(\frac{1}{q_1}-\frac{1}{q_2}\right) \leq \gamma< \frac{n}{q_1}$ and $1 \leq r_1 < r_2 \leq \infty$ $($ or $r_1=r_2=\infty$ $)$.
  \end{theorem}
 In particular, this yields the following outcome within the scope of Lorentz spaces (Definition \ref{Lorentz-space}).
 
\begin{corollary} \label{Adams-Inequality-Lorentz}
    If $1<p_1<p_2<\infty$, $n\left(\frac{1}{p_1}-\frac{1}{p_2}\right) \leq \gamma< \frac{n}{p_1}$ and $1 \leq r_1 < r_2 \leq \infty$ $($ or $r_1=r_2=\infty$ $)$. Then
  $$ \left \Vert I_{\gamma}f \right \Vert_{L^{p_2,r_2}(\mathbb{R}^n, \mu)} \lesssim \left \Vert f \right \Vert_{L^{p_1,r_1}(\mathbb{R}^n, m)}$$
  if and only if the measure $\mu$ satisfies $ \mu(B) \lesssim [m(B)]^{p_2\left(\frac{1}{p_1}-\frac{\gamma}{n}\right)}$ for every ball $B \subset \mathbb{R}^n$.
  \end{corollary} 
 
 \subsection{Function spaces}
 
In this subsection, we fix some notations and recall definitions of certain function spaces required for the subsequent discussion. We begin with Lorentz spaces.

\begin{definition}  \label{Lorentz-space}
 A Lorentz space $L^{p,r}(\Omega, \nu)$ defined over a $\sigma-$finite measure space $(\Omega, \Sigma, \nu)$, consists of all $\nu$-measurable functions on $\Omega$ for which the functional $\|f\|_{L^{p,r}(\Omega, \nu)}$ is finite, where 
\begin{equation*}
\left\Vert f\right\Vert _{L^{p,r}(\Omega, \nu)}:=\left\{ 
\begin{array}{ccc}
\left( \int_{0}^{\infty }\left( t^{\frac{1}{p}}f^{\ast }\left( t\right)
\right) ^{r}\frac{dt}{t} \right)^{1/r} & \text{if} & 0<p<\infty ,0<r<\infty
\\ 
\underset{t>0}{\sup} ~t^{\frac{1}{p}}f^{\ast }\left( t\right) & \text{if} & 0<p\leq
\infty ,r=\infty,%
\end{array}%
\right.
\end{equation*}
and $f^{*}(t) :=\inf\{s \geq 0 : \nu(\{x \in \Omega: |f|>s \}) \leq t \},~t \geq 0,$ is the decreasing (or non-increasing) rearrangement of $f$.
\end{definition}
 We note that $L^{p,p} = L^p$. It is important to emphasize that $\|\cdot\|_{L^{p,r}(\Omega, \nu)}$ is not always a norm, but rather a quasi-norm (see \cite[p. 216]{Sharpley1988}). However, one can define a functional $\|\cdot\|_{L^{(p,r)}(\Omega, \nu)}$ on $L^{p,r}(\Omega, \nu)$ as follows:

\begin{equation*}
\left\Vert f\right\Vert_{L^{\left( p,r\right)}(\Omega, \nu)}:=\left\{
\begin{array}{ccc}
\left( \int_{0}^{\infty}\left( t^{\frac{1}{p}}f^{\ast \ast }\left( t\right)
\right)^{r}\frac{dt}{t} \right)^{1/r} & \text{if} & 0<p<\infty ,0<r<\infty 
\\
\underset{t>0}{\sup}~t^{\frac{1}{p}}f^{\ast \ast }\left( t\right) & \text{if} & 0<p\leq
\infty ,r=\infty,
\end{array}
\right.
\end{equation*}
where the function $f^{**}(t):=\frac{1}{t}\int_0^t f^{*}(t) dt$ is referred to as the maximal average function. This functional is fortunately subadditive. Consequently, $L^{(p,r)}(\Omega, \nu) := \left( L^{p,r}(\Omega, \nu),\|\cdot\|_{L^{(p,r)}(\Omega, \nu)} \right)$ is a normed space for $1<p< \infty$, $1 \leq r \leq \infty$, or $p=r=\infty$.
 Since $f^{*} \leq f^{**}$, we have $L^{(p,r)} \hookrightarrow L^{p,r}$.
Moreover, if $1<p \leq \infty$ and $1 \leq r \leq \infty$, then $L^{p,r} \hookrightarrow L^{(p,r)}$ (see \cite[Lemma 4.5, p. 219]{Sharpley1988}).
 The substitution of $L^{p,1}$ for $L^{p_1}$ on the right-hand side of the inequality (\ref{adams-Inequality}) retains its validity in the limiting case $p_1=p_2=p$ (see \cite{MR4319075,MR3833820}).

 Another important generalization of Lebesgue spaces is the classical Herz space, introduced by C.S. Herz in \cite{Herz}, as a suitable environment for the action of Fourier transform on a Lipschitz class. Although the Herz spaces are defined in various equivalent ways, we adopt the formulation presented in \cite{Garcia94, HTSAA} with a slightly changed notation for our convenience.
 
 Let $\displaystyle \left( \Omega_{t} \right)_{t \in \mathbb{Z}}$ be the dyadic decomposition of $\mathbb{R}^n$ i.e., $\Omega_{t}=\{x\in \mathbb{R}^{n}:2^{t-1} \leq |x|<2^{t}\}$ for $t\in \mathbb{Z}$. We denote $\tilde{\chi}_{\Omega _{t}}=\chi _{\Omega _{t}}$ for $t\in \mathbb{Z}_{+}$,
and $\tilde{\chi}_{\Omega _{-1}}=\chi _{B(0,\frac{1}{2})}$, where $B(0,\frac{1}{2})$ represents the ball centered at the origin with a radius $\frac{1}{2}$ in $\mathbb{R}^n$.

\begin{definition} \label{HLD}
Let $\lambda \in \mathbb{R}$, $0<p,q\leq \infty $ and $\nu$ be a positive measure on $\mathbb{R}^n$.
\begin{itemize}
\item[(i)] The homogeneous Herz space $\dot{K}_{\lambda,q}^{p}(\mathbb{R}^{n}, \nu)$ is defined by
\begin{equation*}
\dot{K}_{\lambda,q}^{p}(\mathbb{R}^{n}, \nu):= \{ f \in L_{loc}^{p}(\mathbb{R}^{n}\setminus \{0\}, \nu): \Vert f \Vert_{\dot{K}_{\lambda,q}^{p}(\mathbb{R}^{n}, \nu)}<\infty\},
\end{equation*}
where $$\Vert f\Vert _{\dot{K}_{\lambda,q}^{p}(\mathbb{R}^{n}, \nu)}:=\left( \sum_{t\in \mathbb{Z}} 2^{t \lambda q}\Vert f\chi _{\Omega _{t}}\Vert _{{L^{p}(\mathbb{R}^n, \nu)}}^{q}\right) ^{\frac{1}{q}}.$$
 \item[(ii)] The inhomogeneous Herz space $K_{\lambda ,q}^{p}(\mathbb{R}^{n}, \nu)$ is defined by 
 \begin{equation*}
K_{\lambda,q}^{p}(\mathbb{R}^{n}, \nu):=\{f\in L_{loc}^{p}(\mathbb{R}^{n}, \nu):\Vert f\Vert
_{K_{\lambda,q}^{p}(\mathbb{R}^{n}, \nu)}<\infty \},
\end{equation*}%
where $$\Vert f\Vert _{K_{\lambda,q}^{p}(\mathbb{R}^{n})}:=\left( \sum_{t=-1}^{\infty
}2^{t \lambda q}\Vert f\tilde{\chi}_{\Omega _{t}}\Vert _{{L^{p}(\mathbb{R}^n, \nu)}}^{q}\right) ^{\frac{%
1}{q}}.$$
\end{itemize}

If $p$ and/or $q$ are infinite, the usual modifications are made.
\end{definition}

It is obvious that $\dot{K}_{0,p}^{p}(\mathbb{R}^{n}, \nu)= K_{0,p}^{p}(\mathbb{R}^{n}, \nu)= L^p(\mathbb{R}^n, \nu)$. In recent years, there has been substantial advancement in the development of Herz spaces, primarily driven by their wide range of applications (see, for instance, \cite{Baernstein,SLPE, SFT,HTH1, PDE,HTH2} and references therein). However,  Herz spaces alone are insufficient to describe some fine properties of functions and operators. Consequently, defining Lorentz-Herz spaces $\dot{HL}_{\lambda,q}^{p,r}(\mathbb{R}^{n}, \nu)$ and $HL_{\lambda,q}^{p,r}(\mathbb{R}^{n}, \nu)$ emerges as a natural progression. These spaces are derived simply by amalgamating Lorentz spaces with Lebesgue sequence spaces—essentially replacing the functionals $\| \cdot\|_{L^{p}(\mathbb{R}^n, \nu)}$ with $\| \cdot\|_{L^{p,r}(\mathbb{R}^n, \nu)}$ in Definition \ref{HLD}. The properties of these spaces, even in more general settings, have been investigated in \cite{my}.

The trace principle for Riesz potentials on Herz spaces and their extensions remains absent from the academic literature. This absence is particularly noteworthy, given the pivotal role that inequalities associated with Riesz potentials play as indispensable tools for estimating functions in terms of their gradients, commonly referred to as Sobolev inequalities. These inequalities are considered as corners stones of the Sobolev theory in partial differential equations.

To establish such estimates within the Herz-type setting, the derivation of trace inequalities is of paramount importance. In the ensuing section, we present a rigorous proof of trace inequalities for both Herz and Lorentz-Herz spaces. Additionally, we engage in a comprehensive discussion on the optimality of specific parametric conditions inherent in these trace inequalities. The resulting trace theorems subsequently facilitate the proof of Sobolev inequalities within Herz space settings, providing succinct estimates for functions in relation to their gradients. As a consequential outcome, we ascertain a Sobolev embedding theorem for Herz-type Sobolev spaces.


\section{Trace Theorems} \label{Section 2}
We commence this section by presenting the trace theorem for Herz spaces. To handle convolution operators with kernels having singularities at the origin, a conventional and widely adopted approach involves the appropriate decomposition of the summation into distinct components, systematically accounting for the presence of the singularity. This well-established technique has found pervasive application in several research papers. Hereafter, if the measure associated with a particular norm is not explicitly mentioned, it is to be understood as the Lebesgue measure on $\mathbb{R}^n$.
\begin{theorem}\label{MT}
Assume that $1 < p_1 < p_2 < \infty$, $1 \leq q_1 \leq q_2 < \infty$, $0 < \gamma < \frac{n}{p_1}$, and $\gamma - \frac{n}{p_1} < \lambda < n - \frac{n}{p_1}$. If $\mu(B) \lesssim [m(B)]^{p_2\left(\frac{1}{p_{1}} - \frac{\gamma}{n}\right)}$ for every ball $B \subset \mathbb{R}^n$, then
\begin{equation} \label{Thm1}
\left \Vert I_{\gamma}f \right \Vert_{\dot{K}_{\lambda,q_2}^{p_2}(\mathbb{R}^n,\mu)} \lesssim \left \Vert f \right \Vert_{\dot{K}_{\lambda,q_1}^{p_1}(\mathbb{R}^n,m)}.
\end{equation}
\end{theorem}

\begin{proof}
Let $f \in \dot{K}_{\lambda,q_1}^{p_1}(\mathbb{R}^n,m)$. Since $0< \frac{q_1}{q_2} \leq 1$, we have
\begin{eqnarray*}
\left \Vert I_{\gamma}f \right \Vert^{q_1}_{\dot{K}_{\lambda,q_2}^{p_2}(\mathbb{R}^n,\mu)} & = & \left[\sum_{t \in \mathbb{Z}} 2^{t \lambda q_2} \left(\int_{\Omega_t} |I_{\gamma}f(x)|^{p_2}d \mu(x)\right)^{\frac{q_2}{p_2}} \right]^{\frac{q_1}{q_2}} \\ & \leq & \left[\sum_{t \in \mathbb{Z}} 2^{t \lambda q_1} \left(\int_{\Omega_t} |I_{\gamma}f(x)|^{p_2}d \mu(x)\right)^{\frac{q_1}{p_2}} \right].
\end{eqnarray*}
By setting $f_s=f\chi_{\Omega_s}$ for every $s \in \mathbb{Z}$, we have $\displaystyle f=\sum_{s \in \mathbb{Z}} f_s$. Using Minkowski's inequality, we get
\begin{eqnarray} \label{three-sums}
\left \Vert I_{\gamma}f \right \Vert_{\dot{K}_{\lambda,q_2}^{p_2}(\mathbb{R}^n,\mu)} & \leq & 
\left[\sum_{t \in \mathbb{Z}} 2^{t \lambda q_1} \left(\sum_{s \in \mathbb{Z}}\left(\int_{\Omega_t} |I_{\gamma}f_s(x)|^{p_2}d \mu(x)\right)^{\frac{1}{p_2}} \right)^{q_1} \right]^{\frac{1}{q_1}} \nonumber \\ & \leq & \left[\sum_{t \in \mathbb{Z}} 2^{t \lambda q_1} \left(\sum_{s \leq t-2}\left(\int_{\Omega_t} |I_{\gamma}f_s(x)|^{p_2}d \mu(x)\right)^{\frac{1}{p_2}} \right)^{q_1} \right]^{\frac{1}{q_1}} \nonumber \\ &
+ & \left[\sum_{t \in \mathbb{Z}} 2^{t \lambda q_1} \left(\sum_{t-1 \leq s \leq t+1}\left(\int_{\Omega_t} |I_{\gamma}f_s(x)|^{p_2}d \mu(x)\right)^{\frac{1}{p_2}} \right)^{q_1} \right]^{\frac{1}{q_1}} \nonumber \\
& + & \left[\sum_{t \in \mathbb{Z}} 2^{t \lambda q_1} \left(\sum_{s \geq t+2}\left(\int_{\Omega_t} |I_{\gamma}f_s(x)|^{p_2}d \mu(x)\right)^{\frac{1}{p_2}} \right)^{q_1} \right]^{\frac{1}{q_1}} \nonumber \\
 & := & E_1+E_2+E_3.
\end{eqnarray}

Now we estimate the terms $E_1$, $E_2$ and $E_3$ one by one. 

{\it Estimation of $E_1$}: For $s \leq t-2$ and a.e. $x \in \Omega_t$, we have

\begin{eqnarray*}
|I_{\gamma}f_s(x)| & \lesssim & 2^{-t(n-\gamma)}\Big|\int_{\mathbb{R}^n}f_s(y)dm(y) \Big| \\ & \leq & 2^{-t(n-\gamma)}\|f_s\|_{L^{p_1}}\|\chi_{\Omega_s}\|_{L^{p'_{1}}}.
\end{eqnarray*}
Thus,
\begin{eqnarray*}
E_1 & = & \left[\sum_{t \in \mathbb{Z}} 2^{t \lambda q_1} \left(\sum_{s \leq t-2}\left(\int_{\Omega_t} |I_{\gamma}f_s(x)|^{p_2}d \mu(x)\right)^{\frac{1}{p_2}} \right)^{q_1} \right]^{\frac{1}{q_1}} \\ & \lesssim & \left[\sum_{t \in \mathbb{Z}} 2^{t \lambda q_1} \left(\sum_{s \leq t-2} 2^{-t(n-\gamma)}\|f_s\|_{L^{p_1}}\|\chi_{\Omega_s}\|_{L^{p'_{1}}} \|\chi_{\Omega_t}\|_{L^{p_2}(\mathbb{R}^n, \mu)} \right)^{q_1} \right]^{\frac{1}{q_1}} \\ & \lesssim & \left[\sum_{t \in \mathbb{Z}} \left(\sum_{s \leq t-2} 2^{s \lambda} \|f_s\|_{L^{p_1}} \cdot 2^{\alpha(t-s)} \right)^{q_1} \right]^{\frac{1}{q_1}}.
\end{eqnarray*}
Here $\alpha=\frac{n}{p_1}-n+ \lambda<0$. Using H\"older's inequality for inside sum and changing order of summations, we get
\begin{eqnarray*}
E_1 & \lesssim & \left[\sum_{t \in \mathbb{Z}} \left(\sum_{s \leq t-2} 2^{s \lambda q_1} \|f_s\|^{q_1}_{L^{p_1}} \cdot 2^{\frac{\alpha q_1}{2}(t-s)} \left\{ \sum_{s \leq t-2} 2^{\frac{\alpha q'_1}{2}(t-s)} \right\}^{\frac{q_1}{q'_1}}\right) \right]^{\frac{1}{q_1}} \\ & \lesssim & \left[\sum_{t \in \mathbb{Z}} \sum_{s \leq t-2} 2^{s \lambda q_1} \|f_s\|^{q_1}_{L^{p_1}} \cdot 2^{\frac{\alpha q_1}{2}(t-s)}  \right]^{\frac{1}{q_1}} \\ & =& \left[\sum_{s \in \mathbb{Z}} 2^{s \lambda q_1} \|f_s\|^{q_1}_{L^{p_1}} \sum_{t \geq s+2} 2^{\frac{\alpha q_1}{2}(t-s)}  \right]^{\frac{1}{q_1}} \\ & \lesssim & \left \Vert f \right \Vert_{\dot{K}_{\lambda,q_1}^{p_1}(\mathbb{R}^n,m)}.
\end{eqnarray*}

{\it Estimation of $E_2$}: Applying Minkowski's inequality and (\ref{adams-Inequality}), we have
\begin{eqnarray*}
E_2 & \leq & \left[\sum_{t \in \mathbb{Z}} 2^{t \lambda q_1} \left(\sum_{t-1 \leq s \leq t+1} \|I_{\gamma} f_s \|_{L^{p_2}(\mathbb{R}^n, \mu)} \right)^{q_1} \right]^{\frac{1}{q_1}} \\ & \leq & \left[\sum_{t \in \mathbb{Z}} 2^{t \lambda q_1} \|I_{\gamma} f_{t-1} \|^{q_1}_{L^{p_2}(\mathbb{R}^n, \mu)} \right]^{\frac{1}{q_1}} + \left[\sum_{t \in \mathbb{Z}} 2^{t \lambda q_1} \|I_{\gamma} f_{t} \|^{q_1}_{L^{p_2}(\mathbb{R}^n, \mu)}\right]^{\frac{1}{q_1}} + \left[\sum_{t \in \mathbb{Z}} 2^{t \lambda q_1} \|I_{\gamma} f_{t+1} \|^{q_1}_{L^{p_2}(\mathbb{R}^n, \mu)}  \right]^{\frac{1}{q_1}} \\ & \lesssim & \left \Vert f \right \Vert_{\dot{K}_{\lambda,q_1}^{p_1}(\mathbb{R}^n,m)}.
\end{eqnarray*}

{\it Estimation of $E_3$}: For $s \geq t+2$ and a.e $x \in \Omega_t$, by using a similar technique as in estimation of $E_1$, we get
\begin{eqnarray*}
|I_{\gamma}f_s(x)| & \lesssim &  2^{-s(n-\gamma)}\|f_s\|_{L^{p_1}}\|\chi_{\Omega_s}\|_{L^{p'_{1}}}.
\end{eqnarray*}
Therefore,
\begin{eqnarray*}
E_3 & = & \left[\sum_{t \in \mathbb{Z}} 2^{t \lambda q_1} \left(\sum_{s \geq t+2}\left(\int_{\Omega_t} |I_{\gamma}f_s(x)|^{p_2}d \mu(x)\right)^{\frac{1}{p_2}} \right)^{q_1} \right]^{\frac{1}{q_1}} \\ & \lesssim & \left[\sum_{t \in \mathbb{Z}} 2^{t \lambda q_1} \left(\sum_{s \geq t+2} 2^{-s(n-\gamma)}\|f_s\|_{L^{p_1}}\|\chi_{\Omega_s}\|_{L^{p'_{1}}} \|\chi_{\Omega_t}\|_{L^{p_2}(\mathbb{R}^n, \mu)} \right)^{q_1} \right]^{\frac{1}{q_1}} \\ & \lesssim & \left[\sum_{t \in \mathbb{Z}} \left(\sum_{s \geq t+2} 2^{s \lambda} \|f_s\|_{L^{p_1}} \cdot 2^{\beta(t-s)} \right)^{q_1} \right]^{\frac{1}{q_1}}.
\end{eqnarray*}
Where $\delta=\frac{n}{p_1}- \gamma+\lambda > 0$. Now use H\"older's inequality for inside sum and then interchange the order of summations, we obtain
\begin{eqnarray*}
E_3 & \lesssim & \left[\sum_{t \in \mathbb{Z}} \left(\sum_{s \geq t+2} 2^{s \lambda q_1} \|f_s\|^{q_1}_{L^{p_1}} \cdot 2^{\frac{\delta q_1}{2}(t-s)} \left\{ \sum_{s \geq t+2} 2^{\frac{\delta q'_1}{2}(t-s)} \right\}^{\frac{q_1}{q'_1}}\right) \right]^{\frac{1}{q_1}} \\ 
& \lesssim & \left[\sum_{t \in \mathbb{Z}} \sum_{s \geq t+2} 2^{s \lambda q_1} \|f_s\|^{q_1}_{L^{p_1}} \cdot 2^{\frac{\delta q_1}{2}(t-s)}  \right]^{\frac{1}{q_1}} \\ 
& =& \left[\sum_{s \in \mathbb{Z}} 2^{s \lambda q_1} \|f_s\|^{q_1}_{L^{p_1}} \sum_{t \leq s-2} 2^{\frac{\delta q_1}{2}(t-s)}  \right]^{\frac{1}{q_1}} \\ 
& \lesssim & \left \Vert f \right \Vert_{\dot{K}_{\lambda,q_1}^{p_1}(\mathbb{R}^n,m)}.
\end{eqnarray*}
This completes the proof.
\end{proof}

In the case of $p_i=q_i$ for $i=1,2$ and $\lambda =0$, the converse of the above theorem holds (see (\ref{adams-Inequality}). However, in general, the question of its converse remains an open inquiry. Nevertheless, we establish a partial answer for a particular set of parameters.
\begin{proposition} \label{Necessity}
Let $p_1,p_2,q_1,q_2$ and $\gamma$ be as in Theorem \ref{MT}. Suppose $p_1 \leq q_1 \leq q_2 \leq p_2$ and $\lambda=0$. 
 If $\left \Vert I_{\gamma}f \right \Vert_{\dot{K}_{\lambda,q_2}^{p_2}(\mathbb{R}^n,\mu)} \lesssim \left \Vert f \right \Vert_{\dot{K}_{\lambda,q_1}^{p_1}(\mathbb{R}^n,m)}$, then $\mu(B) \lesssim [m(B)]^{p_2\left(\frac{1}{p_{1}}-\frac{\gamma}{n}\right)}$ for every ball $B \subset \mathbb{R}^n$.
\end{proposition} 
\begin{proof}
For a given ball $B \subset \mathbb{R}^n$, set $f(x)=  \chi_B(x)$. Then
\begin{eqnarray} \label{remark-eqn1}
\left \Vert f \right \Vert_{\dot{K}_{0,q_1}^{p_1}(\mathbb{R}^n,m)} & = & \left[ \sum_{t \in \mathbb{Z}} \left(\int_{\Omega_t}  \Big| \chi_B(x)\Big|^{p_1} dm(x) \right)^{\frac{q_1}{p_1}} \right]^{\frac{1}{q_1}} \nonumber 
\\ & \leq & \left[ \sum_{t \in \mathbb{Z}}{[m(B \cap \Omega_t)]} \right]^{\frac{1}{p_1}} \nonumber \\ & = & [m(B)]^{\frac{1}{p_1}}.
\end{eqnarray}
Moreover,
\begin{eqnarray}
\left \Vert I_{\gamma}f \right \Vert_{\dot{K}_{0,q_2}^{p_2}(\mathbb{R}^n,\mu)} & = & \left[ \sum_{t \in \mathbb{Z}} \left(\int_{\Omega_t}  \Big| \int_{\mathbb{R}^n}\frac{\chi_B(y)}{|x-y|^{n-\gamma}} dm(y)\Big|^{p_2} d \mu(x) \right)^{\frac{q_2}{p_2}} \right]^{\frac{1}{q_2}} \nonumber \\ & \gtrsim &  \left[ \sum_{t \in \mathbb{Z}} \left(\int_{\Omega_t  \cap B}  \Big| \int_{B}\frac{1}{|x-y|^{n-\gamma}} dm(y)\Big|^{p_2} d \mu(x) \right)^{\frac{q_2}{p_2}} \right]^{\frac{1}{q_2}} \nonumber.
\end{eqnarray}
Since $x,y \in B$, we have $|x-y| \leq 2r$, where $r$ is the radius of the ball. Thus,
$$\frac{1}{|x-y|^{n-\gamma}} \gtrsim \frac{1}{(r^n)^{1-\frac{\gamma}{n}}} \gtrsim [m(B)]^{\frac{\gamma}{n}-1}.$$
Hence,
\begin{eqnarray} \label{remark-eqn2}
\left \Vert I_{\gamma}f \right \Vert_{\dot{K}_{0,q_2}^{p_2}(\mathbb{R}^n,\mu)} & \gtrsim & \left[ \sum_{t \in \mathbb{Z}} [m(B)]^{\frac{q_2\gamma}{n}} [\mu(\Omega_t \cap B)]^{\frac{q_2}{p_2}}\right]^{\frac{1}{q_2}} \nonumber \\ & \gtrsim & [m(B)]^{\frac{\gamma}{n}}\left[ \sum_{t \in \mathbb{Z}} [\mu(\Omega_t \cap B)]\right]^{\frac{1}{p_2}} \nonumber \\ & = & [m(B)]^{\frac{\gamma}{n}} [\mu(B)]^{\frac{1}{p_2}}.
\end{eqnarray}
From (\ref{remark-eqn1}) and (\ref{remark-eqn2}), we get
$\mu(B) \lesssim [m(B)]^{p_2\left(\frac{1}{p_{1}}-\frac{\gamma}{n}\right)}$.
\end{proof}

Next we present the trace inequality for homogeneous Lorentz-Herz spaces. Since the proof follows a similar structure to that of Theorem \ref{MT}, we only provide the necessary steps and point out the differences in the arguments.  

\begin{theorem} \label{MT-LH}
Let $p_1,p_2,q_1,q_2, \mu$ be as in Theorem \ref{MT} and $1 \leq r_1 < r_2 \leq \infty$ (or $r_1 = r_2 = \infty$). Suppose $n\left(\frac{1}{p_1}-\frac{1}{p_2}\right) \leq \gamma< \frac{n}{p_1}$ and $\gamma-\frac{n}{p_1}<\lambda<n-\frac{n}{p_1}$. Then 
$$\left \Vert I_{\gamma}f \right \Vert_{\dot{HL}_{\lambda,q_2}^{p_2,r_2}(\mathbb{R}^{n}, \mu)} \lesssim \left \Vert f \right \Vert_{\dot{HL}_{\lambda,q_1}^{p_1,r_1}(\mathbb{R}^{n}, m)}.$$

%
\end{theorem}
\begin{proof}
Let $f \in \dot{HL}_{\lambda,q_1}^{p_1,r_1}(\mathbb{R}^{n}, m)$. Then, it is easy to see that
\begin{eqnarray*}
\left \Vert I_{\gamma}f \right \Vert_{\dot{HL}_{\lambda,q_2}^{p_2,r_2}(\mathbb{R}^{n}, \mu)} & \leq & \left[\sum_{t \in \mathbb{Z}} 2^{t \lambda q_1} \left \Vert I_{\gamma}f \cdot \chi_{\Omega_t} \right \Vert_{L^{p_2,r_2}(\mathbb{R}^n, \mu)}^{q_1} \right]^{\frac{1}{q_1}}.
\end{eqnarray*}
As before, set $f_s=f\chi_{\Omega_s},~s \in \mathbb{Z}$ and use triangle inequality of the Lorentz norm, we obtain
\begin{eqnarray*}
\left \Vert I_{\gamma}f \right \Vert_{\dot{HL}_{\lambda,q_2}^{p_2,r_2}(\mathbb{R}^{n}, \mu)} & \leq & \left[\sum_{t \in \mathbb{Z}} 2^{t \lambda q_1} \left \Vert I_{\gamma}\left( \sum_{s \in \mathbb{Z}}f_s \right) \cdot \chi_{\Omega_t} \right \Vert_{L^{p_2,r_2}(\mathbb{R}^n, \mu)}^{q_1} \right]^{\frac{1}{q_1}} \\ & \leq & \left[\sum_{t \in \mathbb{Z}} 2^{t \lambda q_1} \left \Vert \sum_{s \in \mathbb{Z}}I_{\gamma}f_s \cdot \chi_{\Omega_t} \right \Vert_{L^{(p_2,r_2)}(\mathbb{R}^n, \mu)}^{q_1} \right]^{\frac{1}{q_1}} \\ & \lesssim & \left[\sum_{t \in \mathbb{Z}} 2^{t \lambda q_1} \left(\sum_{s \in \mathbb{Z}} \left \Vert I_{\gamma}f_s \cdot \chi_{\Omega_t} \right \Vert_{L^{p_2,r_2}(\mathbb{R}^n, \mu)} \right)^{q_1} \right]^{\frac{1}{q_1}}.
\end{eqnarray*}
The interior sum can be broken into three parts, then by the application of Minkowski's inequality, we may write
\begin{eqnarray*}
\left \Vert I_{\gamma}f \right \Vert_{\dot{HL}_{\lambda,q_2}^{L^{p_2,r_2}(\mathbb{R}^n, \mu)}(\mathbb{R}^{n}, \mu)} & \lesssim & \left[\sum_{t \in \mathbb{Z}} 2^{t \lambda q_1} \left(\sum_{s \leq t-2} \left \Vert I_{\gamma}f_s \cdot \chi_{\Omega_t} \right \Vert_{L^{p_2,r_2}(\mathbb{R}^n, \mu)} \right)^{q_1} \right]^{\frac{1}{q_1}} \\ & + & \left[\sum_{t \in \mathbb{Z}} 2^{t \lambda q_1} \left(\sum_{t-1 \leq s \leq t+1} \left \Vert I_{\gamma}f_s \cdot \chi_{\Omega_t} \right \Vert_{L^{p_2,r_2}(\mathbb{R}^n, \mu)} \right)^{q_1} \right]^{\frac{1}{q_1}} \\ & + & \left[\sum_{t \in \mathbb{Z}} 2^{t \lambda q_1} \left(\sum_{s \geq t+2} \left \Vert I_{\gamma}f_s \cdot \chi_{\Omega_t} \right \Vert_{L^{p_2,r_2}(\mathbb{R}^n, \mu)} \right)^{q_1} \right]^{\frac{1}{q_1}} \\ & := & E_1+E_2+E_3.
\end{eqnarray*}
Now, we proceed as in Theorem \ref{MT} except that we use the H\"older inequality for Lorentz spaces and rely on the fact that $\|\chi_{\Omega_s}\|_{L^{p^{\prime}_1, r^{\prime}_1}} \left \Vert \chi_{\Omega_t} \right \Vert_{L^{p_2,r_2}(\mathbb{R}^n, \mu)} \lesssim 2^{[{t(n-\gamma)}+{\frac{n}{p^{\prime}_1}(s-t)}]}$ (which follows along similar lines to \cite[Lemma 3.1.2.1]{my}). Furthermore, the estimation of $E_2$ is based on Corollary \ref{Adams-Inequality-Lorentz}.

\end{proof}
We wrap up this section with the following theorem addressing the limiting case $p_1=p_2=p$. By employing a simple modification of the proof of either Theorem \ref{MT} or Theorem \ref{MT-LH}, in combination with \cite[Theorem 1.2]{MR3833820} (see also \cite[Theorem 3.1]{MR4319075}), we get

\begin{theorem}\label{Open}
Let $1 <p< \infty$, $1 \leq q_1 \leq q_2 < \infty$, $0 < \gamma < \frac{n}{p}$, and $\gamma - \frac{n}{p} < \lambda < n - \frac{n}{p}$. Suppose that for every ball $B \subset \mathbb{R}^n$, we have $\mu(B) \lesssim [m(B)]^{\left(1 - \frac{\gamma p}{n}\right)}$. Then, 
$$\left \Vert I_{\gamma}f \right \Vert_{\dot{K}_{\lambda,q_2}^{p}(\mathbb{R}^n,\mu)} \lesssim \left \Vert f \right \Vert_{\dot{HL}_{\lambda,q_1}^{p,1}(\mathbb{R}^n,m)}.$$
\end{theorem}

The question whether Theorem \ref{Open} holds true when replacing $\dot{HL}_{\lambda,q_1}^{p,1}$ with a space which is not as narrow as this, namely $\dot{HL}_{\lambda,q_1}^{p,r}$ for some $1<r<p$, remains open.

\section{Optimality Conditions}
In this section, we present some examples to illustrate the optimality of certain parametric conditions assumed in Theorem \ref{MT} or Theorem \ref{MT-LH}. To that end, we focus on the case where $n=1$ and $\beta :=p_2\left(\frac{1}{p_{1}}-\gamma\right) \leq 1$. Suppose $\mu$ is the positive Borel measure generated by  $g(x)=x^{\beta-1} \chi_{(0, \infty)}(x)$ on the the Borel sigma algebra $\mathcal{B}$ of $\mathbb{R}$ i.e., $\mu(S)=\int_{S} g~ dm$ for $S \in \mathcal{B}$. We will proceed by working out a few examples using this choice of $\mu$.

\begin{example}
Consider the function $f=\chi_{\Omega_1}$ in $\dot{K}_{\lambda,q_1}^{p_1}(\mathbb{R},m)$. For  $x \in \Omega_t$ and $y \in \Omega_1$, it is evident that $$\frac{1}{|x-y|^{1-\gamma}} \gtrsim \frac{1}{(2+2^t)^{1-\gamma}}.$$ Consequently,
\begin{eqnarray*}
\left \Vert I_{\gamma}f \right \Vert_{\dot{K}_{\lambda,q_2}^{p_2}(\mathbb{R},\mu)} & = & \left[ \sum_{t \in \mathbb{Z}} 2^{t \lambda q_2} \left(\int_{\Omega_t}  \Big| \int_{\mathbb{R}}\frac{\chi_{\Omega_1}(y)}{|x-y|^{1-\gamma}} dm(y)\Big|^{p_2} d \mu(x) \right)^{\frac{q_2}{p_2}} \right]^{\frac{1}{q_2}} \\ & \gtrsim & \left[ \sum_{t \in \mathbb{Z}} \frac{2^{t\lambda q_2} [\mu(\Omega_t)]^{\frac{q_2}{p_2}}}{(2+2^t)^{q_2(1-\gamma)}} \right]^{\frac{1}{q_2}} \\ & \gtrsim & \left[ \sum_{t \leq 0} 2^{t q_2\lambda} [\mu(\Omega_t)]^{\frac{q_2}{p_2}}+\sum_{t \geq 1} 2^{t q_2(\lambda-1+\gamma)} [\mu(\Omega_t)]^{\frac{q_2}{p_2}} \right]^{\frac{1}{q_2}}.
\end{eqnarray*}
Therefore, 
$$\left \Vert I_{\gamma}f \right \Vert_{\dot{K}_{\lambda,q_2}^{p_2}(\mathbb{R},\mu)} \gtrsim \left[\sum_{t \leq 0} 2^{t q_2 \left(\lambda+\frac{1}{p_1}-\gamma \right)}+ \sum_{t \geq 1} 2^{t q_2\left(\lambda-1+\frac{1}{p_{1}} \right)} \right]^{\frac{1}{q_2}}.$$
Thus, for the estimate $\left \Vert I_{\gamma}f \right \Vert_{\dot{K}_{\lambda,q_2}^{p_2}(\mathbb{R},\mu)} \lesssim \left \Vert f \right \Vert_{\dot{K}_{\lambda,q_1}^{p_1}(\mathbb{R},m)}$, it is necessary that $\gamma- \frac{1}{p_1} < \lambda < 1- \frac{1}{p_1}$.
\end{example}
The subsequent example demonstrates the necessity of the condition $q_1 \leq q_2$.

\begin{example} Consider the function $f_k(x)=|x|^{-\left(\lambda+\frac{1}{p_1} \right)} \chi_{\{1<|x|<2^k \}}(x),~k \in \mathbb{N}$. It is not difficult to see that $\|f_k\|_{\dot{K}_{\lambda,q_1}^{p_1}(\mathbb{R},m)} \lesssim k^{\frac{1}{q_1}}$. Moreover,
\begin{eqnarray*}
\|I_{\gamma}f_k\|_{\dot{K}_{\lambda,q_2}^{p_2}(\mathbb{R},\mu)} & = & \left[ \sum_{t \in \mathbb{Z}} 2^{t \lambda q_2} \left(\int_{\Omega_t}  \Big| \int_{\mathbb{R}}\frac{f_k(y)}{|x-y|^{1-\gamma}} dm(y)\Big|^{p_2} d \mu(x) \right)^{\frac{q_2}{p_2}} \right]^{\frac{1}{q_2}} \\ & \gtrsim & \left[ \sum_{t \in \mathbb{Z}} 2^{t \lambda q_2} \left(\int_{\Omega_t}  \Big| \int_{\Omega_t}\frac{f_k(y)}{|x-y|^{1-\gamma}} dm(y)\Big|^{p_2} d \mu(x) \right)^{\frac{q_2}{p_2}} \right]^{\frac{1}{q_2}}.
\end{eqnarray*}
Notice that $\frac{1}{|x-y|^{1-\gamma}} \geq 2^{-(t+1)(1-\gamma)}$ and $f_k(y) \geq 2^{-t\left(\lambda+\frac{1}{p_1} \right)}$ for $x,y \in \Omega_t,~1 \leq t \leq k$. Consequently,
\begin{eqnarray*}
\|I_{\gamma}f_k\|_{\dot{K}_{\lambda,q_2}^{p_2}(\mathbb{R},\mu)} & \gtrsim & \left[ \sum_{1 \leq t \leq k} 2^{q_2\left(t\lambda-(t+1)(1-\gamma)-t\left(\lambda +\frac{1}{p_1}\right)\right)} \left( [m(\Omega_t)]^{p_2}\mu(\Omega_t)\right)^{\frac{q_2}{p_2}} \right]^{\frac{1}{q_2}} \gtrsim \left[ \sum_{1 \leq t \leq k} 1 \right]^{\frac{1}{q_2}} = k^{\frac{1}{q_2}}.
\end{eqnarray*}
Using $\left \Vert I_{\gamma}f \right \Vert_{\dot{K}_{\lambda,q_2}^{p_2}(\mathbb{R},\mu)} \lesssim \left \Vert f \right \Vert_{\dot{K}_{\lambda,q_1}^{p_1}(\mathbb{R},m)}$, we deduce that $k^{\frac{1}{q_2}} \lesssim k^{\frac{1}{q_1}}$. As $k \in \mathbb{N}$ is arbitrary, we must have $q_1 \leq q_2$.
\end{example}
The remaining conditions, namely $p_1 < p_2$ and $\gamma < \frac{n}{p_1}$, are known to be optimal in Lebesgue spaces, and consequently, they stand as optimal conditions for Theorem \ref{MT} (or Theorem \ref{MT-LH}).

\section{Sobolev Inequalities}
The Riesz Potential operator on $\mathbb{R}^n$ classically arises from the $\frac{\gamma}{2}$-th order fractional Laplace equation $(-\Delta)^{\frac{\gamma}{2}}(u)= f$. For $0<\gamma<n$, the function $\mathcal{G}(\gamma)(2\pi |x|)^{-\gamma}$ is the Fourier transform of the function $|x|^{\gamma - n}$ \cite[p. 66]{MR1428685}. Here, the constant $\mathcal{G}(\gamma)$, known as the normalized constant, is given by $$\mathcal{G}(\gamma)=\frac{\pi^{\frac{n}{2}} 2^{\gamma} \Gamma \left(\frac{\gamma}{2}\right)}{\Gamma\left( \frac{n-\alpha}{2} \right)},$$
where $\Gamma$ represents the Euler gamma function. Based on this, it can be readily inferred that the equation $u=\frac{I_{\gamma}(f)}{\mathcal{G}(\gamma)}$ solves the aforementioned equation. Thus, the results in Section \ref{Section 2} indicate that under the conditions of Theorem \ref{MT} (or Theorem \ref{MT-LH}), if $f$ belongs to $\dot{K}_{\lambda,q_1}^{p_1}(\mathbb{R}^n,m)$ $[ \text{or}~ \dot{HL}_{\lambda,q}^{p,r}(\mathbb{R}^{n}, m)]$, then the solution of fractional order equation $(-\Delta)^{\frac{\gamma}{2}}(u)= f$ belongs to $\dot{K}_{\lambda,q_1}^{p_1}(\mathbb{R}^n,\mu)~ [\text{resp.}~ \dot{HL}_{\lambda,q}^{p,r}(\mathbb{R}^{n}, \mu)]$.
 
Another important observation is that if $\mu$ is the Lebesgue measure restricted to Borel sets in $\mathbb{R}^n$ and $\frac{1}{p_2}-\frac{1}{p_1} = \frac{\gamma}{n}$, we can immediately deduce the following Hardy-Littlewood-Sobolev theorem of fractional integration in the context of Lorentz-Herz spaces.
\begin{corollary} \label{HLS-HL}
Let $1<p_1<p_2<\infty$, $1 \leq q_1 \leq q_2<\infty$ and $1 \leq r_1 < r_2 \leq \infty$ $($or $r_1,r_2=\infty$ $)$. Then $$\left \Vert I_{\gamma}f \right \Vert_{\dot{HL}_{\lambda,q_2}^{p_2,r_2}(\mathbb{R}^{n}, m)} \lesssim \left \Vert f \right \Vert_{\dot{HL}_{\lambda,q_1}^{p_1,r_1}(\mathbb{R}^{n}, m)},$$
provided that $\frac{1}{p_1}-\frac{1}{p_2}=\frac{\gamma}{n}$ and $\gamma-\frac{n}{p_1}<\lambda<n-\frac{n}{p_1}$.
\end{corollary}
In particular, if $r_i=p_i,~i=1,2$, we get the corresponding theorem for Herz spaces (cf. \cite{MR1393038}).

Let us recall the definition of Homogeneous Herz-type Sobolev spaces from \cite{MR1451091}. For consistency, we make a slight adjustment to the notation.

\begin{definition}
    Let $1<p<\infty$, $0<\lambda <n \left(1-\frac{1}{p} \right)$, $0<q<\infty$ and $k\in \mathbb{Z}_{+}$. The homogeneous Herz-type Sobolev space $\dot{K}_{\lambda,q}^{p,k}(\mathbb{R}^n)$ is defined by
    $$\dot{K}_{\lambda,q}^{p,k}(\mathbb{R}^n):= \left\{f \in \dot{K}^{p}_{\lambda,q}(\mathbb{R}^n): \text{for}~ |\beta| \leq k, \frac{\partial^{\beta}f}{\partial f^{\beta}}~\text{exists on}~ \mathcal{D}^{\prime}(\mathbb{R}^n) ~\text{and}~\frac{\partial^{\beta}f}{\partial f^{\beta}} \in \dot{K}^{p}_{\lambda,q}(\mathbb{R}^n) \right\},$$ where $\beta=(\beta_1,\beta_2, \cdots, \beta_n) \in \mathbb{Z}_{+}^n$, $\frac{\partial^{0}f}{\partial f^{0}}=f$ and the space is equipped with the functional $$\left \Vert f \right \Vert_{\dot{K}_{\lambda,q}^{p,k}(\mathbb{R}^n)}:=\sum_{|\beta| \leq k} \left \Vert \frac{\partial^{\beta}f}{\partial f^{\beta}}\right \Vert_{\dot{K}^{p}_{\lambda,q}(\mathbb{R}^n)}.$$
\end{definition}

The parameters in the above definition are subjected to specific conditions to ensure the reasonableness of definition, as outlined in \cite{MR1451091}.


\begin{theorem} \label{S-G-N-I}
Let $1<p_1<n, p_2<\infty$, $1 \leq q_1 \leq q_2<\infty$ and $0<\lambda<n-\frac{n}{p_1}$. Suppose that for every ball $B \subset \mathbb{R}^n$, we have $\mu(B) \lesssim [m(B)]^{p_2\left(\frac{1}{p_{1}}-\frac{1}{n}\right)}$. Then 
$$\left \Vert f \right \Vert_{\dot{K}_{\lambda,q_2}^{p_2}(\mathbb{R}^{n}, \mu)} \lesssim \left \Vert \nabla f \right \Vert_{\dot{K}_{\lambda,q_1}^{p_1}(\mathbb{R}^{n}, m)}$$
for every $f \in \dot{K}_{\lambda,q_1}^{p_1,1}(\mathbb{R}^n)$.
\end{theorem}
\begin{proof}
First assume that $g \in \mathcal{D}(\mathbb{R}^n)$, the space of infinitely differentiable functions on $\mathbb{R}^n$ with compact support. Then it is well-known that $|g(x)| \lesssim I_1(|\nabla g|)(x)$ for every $x \in \mathbb{R}^n$. Therefore, by Theorem \ref{MT} and the ideal property of Herz spaces (\cite[Proposition 3.6]{my}), it follows that
\begin{eqnarray} \label{nabla}
\left \Vert g \right \Vert_{\dot{K}_{\lambda,q_2}^{p_2}(\mathbb{R}^n,\mu)} & \lesssim & \left \Vert I_1(|\nabla g|) \right\Vert_{\dot{K}_{\lambda,q_2}^{p_2}(\mathbb{R}^n,\mu)} \nonumber \\ & \lesssim & \left \Vert \nabla g \right \Vert_{\dot{K}_{\lambda,q_1}^{p_1}(\mathbb{R}^n,m)},
\end{eqnarray}
where $\left \Vert \nabla g \right \Vert_{\dot{K}_{\lambda,q_1}^{p_1}(\mathbb{R}^n,m)}= \left \Vert (|\nabla g |) \right \Vert_{\dot{K}_{\lambda,q_1}^{p_1}(\mathbb{R}^n,m)}$ and $|\nabla g |= \sum_{j=1}^n \Big|\frac{\partial g}{\partial x_j} \Big|$. Now let $f \in \dot{K}_{\lambda,q_1}^{p_1,1}(\mathbb{R}^n)$. Then $f \in \dot{K}_{\lambda,q_1}^{p_1}(\mathbb{R}^n)$ and $\frac{\partial f}{\partial x_j} \in \dot{K}_{\lambda,q_1}^{p_1}(\mathbb{R}^n),~j=1,2, \cdots, n$. Moreover, there exists a sequence $\{f_k\}$ in $\mathcal{D}(\mathbb{R}^n)$ such that $f_k \to f$ in $\dot{K}_{\lambda,q_1}^{p_1}(\mathbb{R}^n)$ and  $\frac{\partial f_k}{\partial x_j} \to \frac{\partial f}{\partial x_j}$ in $\dot{K}_{\lambda,q_1}^{p_1}(\mathbb{R}^n)$ \cite[Proposition 2.1]{MR1451091}. Therefore, by equation (\ref{nabla}), we get
$$\left \Vert  f_k-f_l \right \Vert_{\dot{K}_{\lambda,q_2}^{p_2}(\mathbb{R}^n,\mu)} \lesssim \left \Vert \sum_{j=1}^n \Big|\frac{\partial f_k}{\partial x_j} -\frac{\partial f_l}{\partial x_j} \Big| \right \Vert_{\dot{K}_{\lambda,q_1}^{p_1}(\mathbb{R}^n,m)},$$
from which it follows that the sequence $\{f_k\}$ converges to $f$ in $\dot{K}_{\lambda,q_2}^{p_2}(\mathbb{R}^n,\mu)$. This completes the proof.
\end{proof}

The repeated application of the pointwise estimate, $|g(x)| \lesssim I_1(|\nabla g|)(x)$ in combination with semigroup property $I_{\alpha}I_{\beta}=I_{\alpha+\beta}$, ensures that the above inequality holds for higher order Sobolev-type Herz spaces as well. 

\begin{theorem}
Let $k \in \mathbb{N}$, $1<p_1<\frac{n}{k}, p_2<\infty$, $1 \leq q_1 \leq q_2<\infty$, $0<\lambda<n-\frac{n}{p_1}$. If $\mu(B) \lesssim [m(B)]^{p_2\left(\frac{1}{p_{1}}-\frac{k}{n}\right)}$ for every ball $B \subset \mathbb{R}^n$, then 
$$\left \Vert f \right \Vert_{\dot{K}_{\lambda,q_2}^{p_2}(\mathbb{R}^{n}, \mu)} \lesssim \left \Vert \nabla^k f \right \Vert_{\dot{K}_{\lambda,q_1}^{p_1}(\mathbb{R}^{n}, m)}$$
for every $f \in \dot{K}_{\lambda,q_1}^{p_1,k}(\mathbb{R}^n)$.
\end{theorem}


We say $p^{*}$ is the $k$-Sobolev conjugate of $p$ if $\frac{1}{p^{*}}=\frac{1}{p}-\frac{k}{n}$, where $k$ is a positive integer. We simply write Sobolev conjugate for 1-Sobolev conjugate.
Putting $\mu=m$ in the above theorem, we get the following Sobolev embedding theorem for Herz-type Sobolev spaces.
\begin{corollary} Let $k \in \mathbb{N}$, $1<p<\frac{n}{k}$, $1 \leq q < \infty$, $0<\lambda<n-\frac{n}{p}$ and $p^{*}$ be the $k$-Sobolev conjuage of $p$. Then 
$$\left \Vert f \right \Vert_{\dot{K}_{\lambda,q}^{p^{*}}(\mathbb{R}^{n}, m)} \lesssim \left \Vert \nabla^k f \right \Vert_{\dot{K}_{\lambda,q}^{p}(\mathbb{R}^{n}, m)}$$
for every $f \in \dot{K}_{\lambda,q}^{p,k}(\mathbb{R}^n)$. In  particular, $\dot{K}_{\lambda,q}^{p,k}(\mathbb{R}^n) \hookrightarrow \dot{K}_{\lambda,q}^{p^{*}}(\mathbb{R}^{n})$.
\end{corollary}
Finally, we prove the following Gagliardo-Nirenberg-Sobolev (GNS) inequality in the setting of Herz spaces.
\begin{theorem}
Let $0 \leq \theta \leq 1$, $1<p_0<n$, $1 \leq p_0,p_1 < p_2<\infty$, $1 \leq q_0, q_1 < q_2<\infty$ and $0<\lambda<n-\frac{n}{p_0}$. Suppose that $\mu(B) \lesssim [m(B)]^{p_2\left(\frac{1}{p_{0}}-\frac{1}{n}\right)}$ for every ball $B \subset \mathbb{R}^n$. Then $$\left \Vert f \right \Vert_{\dot{K}_{\lambda,q}^{p}(\mathbb{R}^n, \mu)} \leq \left \Vert f \right \Vert^{1-\theta}_{\dot{K}_{\lambda,q_1}^{p_1}(\mathbb{R}^n, \mu)} \left \Vert \nabla f \right \Vert^{\theta}_{\dot{K}_{\lambda,q_0}^{p_0}(\mathbb{R}^n, m)}$$ for every $f \in \dot{K}_{\lambda,q_1}^{p_1}(\mathbb{R}^n, \mu) \cap \dot{K}_{\lambda,q_0}^{p_0,1}(\mathbb{R}^n,m)$, provided that $\frac{1}{q}=\frac{1-\theta}{q_1} + \frac{\theta}{q_2}$ and $\frac{1}{p}=\frac{1-\theta}{p_1} + \frac{\theta}{p_2}$.
\end{theorem}
\begin{proof}
Let $0\leq \theta \leq 1$ and $1\leq p_1<p_2 \leq \infty$. Using the interpolation inequality $\|f\|_{L^p(\mathbb{R}^n, \nu)} \leq \|f\|^{1-\theta}_{L^{p_1}(\mathbb{R}^n, \nu)} \|f\|^{\theta}_{L^{p_2}(\mathbb{R}^n, \nu)}$, which holds for any positive measure $\nu$ on $\mathbb{R}^n$, provided that $\frac{1}{p}=\frac{1-\theta}{p_1} + \frac{\theta}{p_2}$, and the H\"older inequality, we obtain
\begin{eqnarray*}
\sum_{t \in \mathbb{Z}}2^{t \lambda q} \|f \chi_{\Omega_t}\|^q_{L^p(\mathbb{R}^n, \nu)} 
& \leq & \left(\sum_{t \in \mathbb{Z}}2^{t \lambda q_1}\|f \chi_{\Omega_t}\|^{q_1}_{L^{p_1}(\mathbb{R}^n, \nu)} \right)^{\frac{q(1-\theta)}{q_1}} \left(\sum_{t \in \mathbb{Z}}2^{t \lambda q_2}\|f \chi_{\Omega_t}\|^{q_2}_{L^{p_2}(\mathbb{R}^n, \nu)} \right)^{\frac{q \theta}{q_2}}.
\end{eqnarray*}
Consequently, $$\left \Vert f \right \Vert_{\dot{K}_{\lambda,q}^{p}(\mathbb{R}^n, \nu)} \leq \left \Vert f \right \Vert^{1-\theta}_{\dot{K}_{\lambda,q_1}^{p_1}(\mathbb{R}^n, \nu)} \left \Vert f \right \Vert^{\theta}_{\dot{K}_{\lambda,q_2}^{p_2}(\mathbb{R}^n, \nu)}.$$
Using Theorem \ref{S-G-N-I}, it follows that
$$\left \Vert f \right \Vert_{\dot{K}_{\lambda,q}^{p}(\mathbb{R}^n, \mu)} \leq \left \Vert f \right \Vert^{1-\theta}_{\dot{K}_{\lambda,q_1}^{p_1}(\mathbb{R}^n, \mu)} \left \Vert \nabla f \right \Vert^{\theta}_{\dot{K}_{\lambda,q_0}^{p_0}(\mathbb{R}^n, m)}.$$
\end{proof}
\begin{remark}
\begin{itemize}
\item[(i)] If $q \leq p$, the definition of $\dot{K}_{\lambda,q}^{p,k}(\mathbb{R}^n)$ remains reasonable even when $\lambda=0$. Evidently, all the aforementioned results (Theorem \ref{S-G-N-I} onwards) are still true when $\lambda=0$ and $q_i \leq p_i~ (i=0,1,2)$.
\item[(ii)] If $\lambda=0$, $\mu=m$, $p_i=q_i$ for $i=0,1,2$, and $p_2$ is the Sobolev conjugate of $p_0$, we obtain the GNS inequality for the Lebesgue spaces (see \cite[Section 1]{MR3989143}).
\end{itemize}
\end{remark}

%
%
%
%
%


\section*{Acknowledgements}

The first author (M. Ashraf Bhat) is thankful to Prime Minister's Research
Fellowship (PMRF) program for the fellowship (PMRF ID: 2901480).

\section*{Declarations}
\subsection*{Data Availability}

No data were used to support this study.

\subsection*{Conflicts of Interest}

The authors declare that there is no conflict of interest.

\subsection*{Funding}

This research work did not receive any external funding.

\subsection*{Authors' Contributions}
The authors contributed equally to the writing of this paper. All the authors read and approved the manuscript.


\begin{thebibliography}{99}

\bibitem{MR336316} Adams, D. R., {\it A trace inequality for generalized potentials}, Studia Math. {\bf 48} (1973), 99--105.
%

\bibitem{MR417774} Adams, D. R., {\it On the existence of capacitary strong type estimates in {$R^{n}$}}, Ark. Mat. {\bf 14} (1976), 125--140.

\bibitem{MR287301} Adams, D. R., {\it Traces of potentials arising from translation invariant operators}, Ann. Scuola Norm. Sup. Pisa Cl. Sci. (3) {\bf 25} (1971), 203--217.

\bibitem{Baernstein} Baernstein, II, A. and Sawyer, E. T., {\it Embedding and multiplier theorems for {$H^{p}({\bf R}^n)$}}, Mem. Amer. Math. Soc. {\bf 53} (1985), iv+82.


\bibitem{Sharpley1988} Bennett, C. and Sharpley, R., {\it Interpolation of operators}, vol. 129, Pure and Applied Mathematics, Academic Press, Inc., Boston, MA, 1988.

\bibitem{my} Bhat, M. A., Kolwicz, P., Kosuru, G. S. R., {\it K\"othe-Herz Spaces: The Amalgam-Type Spaces of Infinite Direct Sums}, Preprint (2023), 1--32.
Available at: https://doi.org/10.48550/arXiv.2209.05897


\bibitem{MR4327616} de Almeida, M. F. and Lima, L. S. M., {\it Adams' trace principle in {M}orrey-{L}orentz spaces on {$\beta$}-{H}ausdorff dimensional surfaces}, Ann. Fenn. Math. {\bf 46} (2021), 1161--1177.
\bibitem{SLPE} Drihem, D., {\it Semilinear parabolic equations in Herz spaces}, Appl. Anal. (2022).




\bibitem{MR2028180} Edmunds, D. E., Kokilashvili, V. and Meskhi, A., {\it A trace inequality for generalized potentials in {L}ebesgue spaces with variable exponent}, J. Funct. Spaces Appl. {\bf 2} (2004), 55--69.
\bibitem{SFT} Feichtinger, H. G. and Weisz, F., {\it Herz spaces and summability of Fourier transforms}, Math. Nachr. {\bf 281} (2008), 309--324.




\bibitem{Garcia94}  García-Cuerva, J. and Herrero, M.-J. L., {\it A theory of Hardy spaces associated to the Herz spaces}, Proc. London Math. Soc. (3) {\bf 69} (1994), 605-–628.


\bibitem{Herz} Herz, C. S., {\it Lipschitz spaces and {B}ernstein's theorem on absolutely convergent {F}ourier transforms}, J. Math. Mech. {\bf 18} (1968/69), 283--323.


\bibitem{MR4319075} Imerlishvili, G. and Meskhi, A., {\it A note on the trace inequality for {R}iesz potentials}, Georgian Math. J. {\bf 28} (2021), 739--745.
\bibitem{MR3833820} Korobkov, M. V. and Kristensen, J., {\it The trace theorem, the {L}uzin {$N$}- and {M}orse-{S}ard properties for the sharp case of {S}obolev-{L}orentz mappings}, J. Geom. Anal. {\bf 28} (2018), 2834--2856.
\bibitem{MR3718049} Liu, L. and Xiao, J., {\it A trace law for the {H}ardy-{M}orrey-{S}obolev space}, J. Funct. Anal. {\bf 274} (2018), 80--120.

\bibitem{MR1393038} Lu, S. and Yang, D., {\it Hardy-{L}ittlewood-{S}obolev theorems of fractional integration on {H}erz-type spaces and its applications}, Canad. J. Math. {\bf 48} (1996), 363--380.
\bibitem{MR1451091} Lu, S. and Yang, D., {\it Herz-type {S}obolev and {B}essel potential spaces and their applications}, Sci. China Ser. A {\bf 40} (1997), 113--129.
\bibitem{HTSAA}  Lu, S., Yang, D., and Hu, G., {\it Herz type spaces and their applications. essay}, Beijing: Science Press, 2008.


\bibitem{HTH1} Lu, S. and Yang, D., {\it The weighted
Herz-type Hardy space and its applications} , Sci. China Ser. A {\bf 38} (1995), 662--673.



\bibitem{MR1428685} Mizuta, Y., {\it Potential theory in {E}uclidean spaces}, vol. 6, GAKUTO International Series. Mathematical Sciences and Applications, Gakk\={o}tosho Co., Ltd., Tokyo, 1996.

\bibitem{MR3989143} Nguyen, V. H., {\it The sharp {G}agliardo-{N}irenberg-{S}obolev inequality in quantitative form}, J. Funct. Anal. {\bf 277} (2019), 2179--2208.
\bibitem{PDE} Ragusa, M. A., {\it Homogeneous {H}erz spaces and regularity results}, Nonlinear Anal. {\bf 71} (2009), e1909--e1914.


\bibitem{HTH2} Wang, H. and Liu, Z., {\it The Herz-type Hardy
spaces with variable exponent and their applications}, Taiwanese J. Math. {\bf 16} (2012), 1363--1389.




\bibitem{MR3836125} Xiao, J., {\it A trace problem for associate {M}orrey potentials}, Adv. Nonlinear Anal. {\bf 7} (2018), 407--424.


\end{thebibliography}
\end{document}